\documentclass[twoside,11pt]{article}

\usepackage[english]{babel}
\usepackage{amsmath}
\usepackage{amsthm}
\usepackage{amsfonts}
\usepackage{amssymb}
\usepackage{float}
\usepackage{paralist}  
\allowdisplaybreaks[4]

       \theoremstyle{plain}

       \newtheorem{theorem}{\bf Theorem}[section]

       \newtheorem{lemma}[theorem]{\bf Lemma}
       \newtheorem{remark}[theorem]{Remark}

       \numberwithin{equation}{section}

\newcommand{\om}{\Omega}
\newcommand{\R}{\mathbb{R}}
\newcommand{\C}{\mathcal{C}}
\newcommand{\T}{\widehat{T}}

\begin{document}

\title{{\Large Global solutions to the Nernst-Planck-Euler system on bounded domain}
}
 
\author{{Dapeng Du$^a$, Jingyu Li$^a$, Yansheng Ma$^a$\thanks{Corresponding author: may538@nenu.edu.cn}, Ruyi Pang$^a$}\\[2mm]
\small\it$^a$School of Mathematics and Statistics, Northeast Normal University,\\
\small\it   Changchun 130024, China}

\date{}

\maketitle

\begin{quote}
\small \textbf{Abstract}: We show that the Nernst-Planck-Euler system, which models ionic electrodiffusion in fluids, has global strong solutions for arbitrarily large data in the two dimensional bounded domains.  The assumption on species is either there are two species or the diffusivities and the absolute values of ionic valences are the same if the species are  arbitrarily many. In particular, the boundary  conditions for the ions are allowed to be inhomogeneous. The proof is based on the energy estimates,  integration along the characteristic line and the regularity theory of elliptic and parabolic equations.


\indent \textbf{AMS (2010) Subject Classification}: 35Q35, 35Q31, 76B03


\end{quote}

\section{Introduction}
In this paper, we consider the following Nernst-Planck-Euler (NPE) system:
\begin{equation}\label{1.1}
\partial_{t}c_{i}+u\nabla c_{i}=D_{i}\text{div}(\nabla c_{i}+z_{i}c_{i}\nabla\Phi),\ \ i=1,\cdots,N,
\end{equation}
\begin{equation}\label{1.2}
-\varepsilon\Delta\Phi=\sum_{i=1}^{N} z_{i}c_{i}=\rho,
\end{equation}
\begin{equation}\label{1.3}
\partial_{t}u+u\nabla u+\nabla p=-K\rho\nabla\Phi,
\end{equation}
\begin{equation}\label{1.4}
\nabla\cdot u=0,
\end{equation}
which models the evolution of ions in inviscid incompressible  fluid \cite{Ru}. Here $c_{i}$ is the concentration of the i-th ionic species, $0\neq z_{i}\in \R$ is the corresponding valence, $\Phi$ is the electrical potential, $\rho$ is the charge density, $u$ is the fluid velocity, $p$ is the fluid pressure, $D_{i}>0$ are diffusion coefficients, $\varepsilon>0$ is a constant proportional to the square of Debye length, and $K>0$ is the Boltzmann constant.

Let $\Omega\subset \mathbb{R}^{2}$ be a smooth bounded domain.  The unknowns $c_{i}$ and $\Phi$ satisfy the  following inhomogeneous Dirichlet boundary conditions:
\begin{equation}\label{1.5}
c_{i}(x,t)\mid_{\partial\Omega}=\gamma_{i}(x)\geq0,
\end{equation}
\begin{equation}\label{1.6}
\Phi(x,t)\mid_{\partial\Omega}=h(x),
\end{equation}
where $\gamma_{i}(x)$ and $h(x)$ are given smooth functions independent of $t$. The velocity $u$ satisfies homogeneous Dirichlet boundary condition:
\begin{equation}\label{1.7}
u\mid_{\partial\Omega}=0.
\end{equation}

The NPE system is the invisid case of the Nernst-Planck-Navier-Stokes (NPNS) system. There are lots of interesting analytical works regarding the NPNS system. Liu and Wang \cite{Liu} showed the global existence of weak solutions to the Cauchy problem in $\R^3$. Schmuck \cite{Schmuck} proved global existence of weak solutions to the NPNS system with blocking (no-flux) boundary conditions for the ionic concentrations and homogeneous Neumann boundary conditions for the electric potential in two and three dimensions. Ryham \cite{Ryham} generalized this work to the case of homogeneous Dirichlet boundary conditions for the electric potential.

It turns out that boundary conditions for ionic concentrations play an essential role in the dynamics of NPNS system.  Bothe, Fischer and Saal \cite{Bothe} proved global existence and stability of strong solutions in two dimensions, with blocking (no-flux) boundary conditions for the ionic concentrations.
In a series of papers \cite{CI,CIL20,CIL20-far,CIL20-neutral,CIL22}, Constantin, Ignatova and Lee studied the existence and stability of the NPNS system systematically. In \cite{CI}, Constantin and Ignatova proved global well-posedness of strong solutions in two dimensions with inhomogeneous Dirichlet boundary conditions for the ionic concentrations  and electric potential.  The inhomogeneous boundary conditions make things much more complicated,  since one cannot get estimates by directly multiplying
\eqref{1.1} with $c_i$.
In \cite{CIL20}, Constantin, Ignatova and Lee got the three dimensional global existence near the equilibrium. For the equilibrium, the velocity is zero.
Later they \cite{CIL20-far} proved the global well-posedness in three dimensions.
The assumptions on ionic concentrations are either there are two ionic species or all diffusivities are equal and the magnitudes of valences are the same in the case of more than two species.
  These assumptions are also the ones commonly used in the  literature.
In \cite {CIL20-neutral} they analyzed the phenomena of interior electroneutriality in two and three dimensions.
Very recently, they are able to get stability result even for equilibrium with non-zero velocity \cite{CIL22}.
When the electric potential satisfies the Robin condition, the existence are obtained by Lee \cite{L}.

 Comparing with the rich literature for NPNS, the results for NPE are much less. In the case of two dimensional periodic boundary conditions, Ignatova and Shu \cite{IS} proved the global existence and uniqueness of strong solutions to the NPE system.
In this paper we get the global well-posednesd in the case of bounded domain. The  boundary conditions for the ionic concentrations are Dirichlet and allowed to be inhomogeneous.

To state our result more precisely, we denote by $W^{k,p}(\Omega)$ the Sobolev space whose norm is given by $\|f\|_{W^{k,p}(\Omega)}:=\sum_{j=0}^k\|\partial_x^jf\|_{L^p(\Omega)}$; $W^{k,p}_0(\Omega)$ denotes the closure of $\C_0^\infty(\Omega)$ in the norm $\|\cdot\|_{W^{k,p}(\Omega)}$; and $W^{2,1}_p(Q_T)$ is the time-dependent Sobolev space on $Q_{T}=\Omega\times(0,T)$ with norm $\|f\|_{W^{2,1}_p(Q_T)}:=\sum_{|\alpha|+2r\leq2}\|\partial^\alpha_x \partial^r_tf\|_{L^p(Q_T)}$.

\begin{theorem}\label{thm-1}
 Assume that the ionic species satisfy one of the following two conditions:

  i)  There are two ionic species.

  ii)  For arbitrarily many ionic species, the diffusivities  and the absolute values of valences are the same,
i.e.  $D_i\equiv D$, $|z_i |\equiv z$, $i=1,\cdots,N$.

Suppose that $0\leq c_i(0) \in L^p(\Omega)$ with $p>2$, $i=1,\cdots,N$, and that $u(0)\in \C^{1,\alpha}(\bar{\Omega})$ with $0<\alpha <\frac{1}{2}-\frac{1}{p}$.
Let $\frac{1}{q}=\frac{1}{2}(\frac{3}{2}+\frac{1}{p}
+\alpha)<1.$ Then for any $T>0$ the Nernst-Planck-Euler system \eqref{1.1}-\eqref{1.7} has a unique strong solution satisfying
\[c_{i}\in W_q^{2,1}(Q_T) \ \text{ and } \ u\in \C([0,T]; \C^{1,\alpha}(\bar{\Omega})). \]

\end{theorem}

\begin{remark} Due to the existence of inhomogeneous boundary condition,  the norms of the solutions may go to infinity as time goes to infinity.
\end{remark}
Mathematically, the NPE system is the coupling of a semilinear parabolic system and the incompressible Euler equations.  It's well-known that the Euler equations  is globally well-posed in 2D \cite{Marchioro94}. If we omit the velocities in the equations of the ions, then the energy conservation yields global regularity even in 4D \cite{Choi}.  Combining these two facts together, we can expect global well-posedness in 2D for the NPE system.

The proof follows the typical way in the regularity theory.  We first get local well-posedness, which is more or less  standard for the current system.  Then we derive good a priori estimates,  the key for the global well-posedness.  For the ions, we use energy integration. Regarding the velocity, we employ the  method of integration along the characteristic line. Similar to the NPNS system,  the inhomogeneous boundary condition for the ions make the  arguments much more complicated. Finally, the global existence can be deduced from the local well-posedness and good a priori estimates.

This paper is organized as follows.  In section \ref{sec-2} we  establish the local well-posedness.  Section \ref{sec-3} is devoted to the a priori estimates. Finally we derive the global well-posedess in section \ref{sec-4}.

\section{Local well-posedess}
\label{sec-2} In this section, we establish the local well-posedness of system \eqref{1.1}-\eqref{1.7} using the Schauder fixed point theorem.  The idea is from Kato \cite{Ka}. Define the vorticity $\omega=\text{curl } u$, then $\omega$ satisfies
\begin{equation}\label{omega}
\partial_t\omega+u\nabla\omega=-K\nabla^\bot\rho\cdot\nabla\Phi,
\end{equation}
with $\nabla^\bot:=(-\partial_y,\partial_x)$. We first introduce a function space. Define
$$X:=\{\varphi\mid \ \|\varphi(t)\|_{L^{\infty}}\leq Ct^{-\frac{1}{p}},\|\nabla\varphi(t)\|_{\C^{1,\alpha}}\leq Ct^{-\frac{1}{2}-\frac{1}{p}-\alpha},\|\varphi(t)\|_{L^{p}}
+\|\varphi\|_{W^{2,1}_q(Q_T)}\leq C\},$$
$$\|\varphi\|_{X}:=\sup_{t\in(0,T)}\left[t^{\frac{1}{p}}\|\varphi(t)\|_{L^{\infty}}
+t^{\frac{1}{2}+\frac{1}{p}+\alpha}\|\varphi(t)\|_{\C^{1,\alpha}}+\|\varphi(t)
\|_{L^{p}}\right]
+\|\varphi\|_{W^{2,1}_q(Q_T)},$$
where $\alpha, p, q$ are  the same as  the ones in Theorem \ref{thm-1}. Since we are doing local  solutions, throughout this section $T$ is $O(1)$ or smaller.
Now we introduce some facts about the fundamental solution to the heat equation  in bounded domain. The estimate is essentially standard in  fundamental solution  literature.  We put the proof here for the sake of completeness.

\begin{lemma}
\label{lemma2.1} Assume that $\Omega$ is smooth, then
\begin{equation}
\mid D_{x}^{(k)}H(x,t;y)\mid\leq Ct^{-\frac{2 +\mid k \mid }{2}}e^{-\frac{\mid x-y\mid^{2}}{16t}},\qquad C=C(\Omega,k),
\label{H-xty}
\end{equation} where $H(x,t;y)$ is the fundamental solution of the heat equation.
\end{lemma}
\begin{proof}
 It's well-known  that $H$  enjoys the following estimates ( for instance, Theorem 1 in \cite{Ar67}):
\begin{equation}
0 \leq H(x,t;y) \leq C
 e^{-\frac{\mid x-y\mid^{2}}{4t}}.
 \end{equation}
Note  that this Lemma is  trivial if
$t$ is $O(1)$,  so we only consider the case $t$ is  small.
Fix $\overline{x}\in\Omega$, $0<\overline{t}<T$, we divide the proof into two cases.

(i) Interior:
\begin{equation*}
\textrm{dist}(\overline{x},\partial\Omega)\geq\sqrt{\bar{t}}.
\end{equation*}
Let $Q_{\sqrt{\overline{t}}}:=B_{\sqrt{\overline{t}}}(\overline{x})\times(\frac{\overline{t}}{2},
\overline{t})$, then $Q_{\sqrt{\overline{t}}}\subset Q_{T}$. Do the change of variable :
\begin{equation*}
s=\frac{t}{\overline{t}},z=\frac{x-\overline{x}}{\sqrt{\overline{t}}}.
\end{equation*}
Then
$H(\overline{x}+\sqrt{\overline{t}}z,\overline{t}s;y)$ solves the heat equation in $B_{1}\times(\frac{1}{2},1)$. The standard parabolic smoothing estimates imply
\begin{equation*}
\begin{split}
\left|D_{z}H(\overline{x}+\sqrt{\overline{t}}z,\overline{t}s;y)\right|_{z=0,s=1}
&\leq \textrm{sup}_{B_{1}\times(\frac{1}{2},1)} H(\overline{x}+\sqrt{\overline{t}}z,\overline{t}s;y)\\
&=\textrm{sup}_{Q_{\sqrt{\overline{t}}}}H(x,t;y).
\end{split}
\end{equation*}
Thus
\begin{equation}
\begin{split}
\mid D_{x}H(\overline{x},\overline{t};y)\mid
&=\overline{t}^{-\frac{1}{2}}\mid D_{z}H(\overline{x}+\sqrt{\overline{t}}z,
\overline{t}s;y)\mid\\
&\leq\overline{t}^{-\frac{1}{2}}\textrm{sup}_{Q_{\sqrt{\overline{t}}}} H(x,t;y).
\end{split}
\end{equation}
Since
\begin{equation*}
H(x,t;y)\leq Ct^{-1}e^{-\frac{|x-y|^{2}}{4t}},\\
\end{equation*}
we have
\begin{equation}
\begin{split}
\mid D_{x}H(\overline{x},\overline{t};y)\mid&\leq\left\{\begin{array}{ll}
C\overline{t}^{-\frac{3}{2}}e^{-\frac{(\mid \overline{x}-y\mid-\sqrt{\overline{t}})^{2}}{4\overline{t}}},& \mid\overline{x}-y\mid\geq2\sqrt{\overline{t}},\\
C\overline{t}^{-\frac{3}{2}},&\mid\overline{x}-y\mid<2\sqrt{\overline{t}},
\end{array}\right.\\
&\leq C\overline{t}^{-\frac{3}{2}}e^{-\frac{\mid\overline{x}-y\mid^{2}}{16\overline{t}}}.
\end{split}
\end{equation}

(ii) Near the boundary:
\begin{equation*}
\textrm{dist}(\overline{x},\partial\Omega)\leq\sqrt{\overline{t}}.
\end{equation*}
Define $\overline{x}_{1}$ to be
\begin{equation*}
\textrm{dist}(\overline{x},\partial\Omega)=\mid\overline{x}-\overline{x}_{1}\mid.
\end{equation*}
Let
\begin{equation}
Q_{b}=(B_{2\sqrt{\overline{t}}}(\overline{x})\cap\Omega)
\times(\frac{\overline{t}}{2},\overline{t}).
\end{equation}
Since $H(x,t;y)\mid_{\partial\Omega}=0$ and
$H$ is smooth up to the boundary, we can get \eqref{H-xty} by the standard flattening procedure.
\end{proof}

\begin{lemma}\label{lemma2.2} Assume that $p>2$, $c_{i}(0)\in L^{p}$, then $\int_{\Omega}H(x,t;y)c_{i}(0)(y)dy\in X.$
\end{lemma}

\begin{proof} By Lemma \ref{lemma2.1},
\begin{equation*}
\begin{split}
\left|\int_{\Omega}H(x,t;y)c_{i}(0)(y)dy\right|
&\leq\|H(x,t;y)\|_{L^{q}}\| c_{i}(0)\|_{L^{p}}\\
&\leq C\left(t^{-q}\int_{\Omega}e^{-\frac{q|x-y|^{2}}{16t}}dy\right)^{\frac{1}{q}}\cdot\|
c_{i}(0)\|_{L^{p}}\\
&\leq Ct^{-1+\frac{1}{q}}\| c_{i}(0)\|_{L^{p}}\\
&=Ct^{-\frac{1}{p}}\| c_{i}(0)\|_{L^{p}},
\end{split}
\end{equation*} and
\begin{equation*}
\begin{split}
\left|D_{x}\int_{\Omega}H(x,t;y)c_{i}(0)(y)dy\right|
&=\left|\int_{\Omega}D_{x}H(x,t;y)c_{i}(0)(y)dy\right|\\
&\leq\|D_{x}H(x,t;y)\|_{L^{q}}\| c_{i}(0)\|_{L^{p}}\\
&\leq C\left(\int_{\Omega}t^{-\frac{3q}{2}}e^{-\frac{q(x-y)^{2}}{16t}}dy
\right)^{\frac{1}{q}}\cdot\|
c_{i}(0)\|_{L^{p}}\\
&\leq Ct^{-\frac{3}{2}+\frac{1}{q}}\| c_{i}(0)\|_{L^{p}}\\&
=Ct^{-\frac{1}{2}-\frac{1}{p}}\| c_{i}(0)\|_{L^{p}},
\end{split}
\end{equation*}where $\frac{1}{q}+\frac{1}{p}=1$. Similarly, we can get the estimates for other norms.
\end{proof}

Now we are ready to give the local existence theorem.
\begin{theorem}\label{local-wellposedness}
Assume that $c_i(0) \in L^p(\Omega)$ with $p>2$, $i=1,\cdots,N$, and that $u(0)\in \C^{1,\alpha}(\bar{\Omega})$ with $0<\alpha<\frac{1}{2}-\frac{1}{p}$. Then there exists $T>0$, such that system \eqref{1.1}-\eqref{1.7} has a unique solution $(c_{i},u)$ satisfying
\begin{equation}\label{regularity}
c_{i}\in X \ , \ u\in \C^\alpha(Q_T)
\ and \  u \in   \C([0,T]; \C^{1, \alpha }(\bar{\Omega})).\end{equation}
 The lifespan $T$  only depends on the norms of the initial data.
\end{theorem}

\begin{proof} \emph{Step 1.  Definition of the mapping.} Fix $M, T>0$ to be chosen later. Consider the closed set
\[\begin{split}B_M:=\{(\varphi,\psi)\ | \ &\varphi \in X,
\psi\in \C([0,T];\C^{1,\alpha}(\bar{\Omega}))
\cap \C^\alpha(Q_T);\\&
\|\varphi\|_X+ \|\psi\|_{\C([0,T];\C^{1,\alpha}(\bar{\Omega}))}
+ \|\psi\|_{\C^\alpha(Q_T)}\leq M\}.\end{split}\]
Given $(c_i,u)\in B_M$ with $\nabla\cdot u=0$, $c_{i}(x,t)\mid_{\partial\Omega}=\gamma_{i}(x)$ and $u\mid_{\partial\Omega}=0$, define the mapping $\Lambda$ by
$$\Lambda(c_{i},u)=(d_{i},v),$$
where $(d_{i},v)$ satisfy
\begin{equation}\label{2.2}
\partial_td_{i} -D_{i}\Delta d_{i}=D_{i}z_{i}\text{div}(c_{i}\nabla\Phi)-u\nabla c_{i},
\end{equation}
\begin{equation}\label{2.3}
-\varepsilon\Delta\Phi=\sum_{i=1}^{N} z_{i}c_{i}=\rho,
\end{equation}
\begin{equation}\label{2.4}
\partial_t\omega+u\nabla\omega=-K\nabla^\bot\rho\cdot\nabla\Phi,
\end{equation}
\begin{equation}\label{2.5}
\omega= \text{curl } v,
\end{equation}
and the initial boundary conditions
\begin{equation}
d_{i}\mid_{\partial\Omega}=\gamma_{i}(x), \ d_i(0)=c_i(0),
\end{equation}
\begin{equation}
\Phi\mid_{\partial\Omega}=h(x),
\end{equation}
\begin{equation}
v\mid_{\partial\Omega}=0, \ v(0)=u(0).
\end{equation}

\emph{Step 2. Invariantness.}
In this step we shall show that $\Lambda B_M \subset B_M$. Define \[\begin{split}g:=D_{i}z_{i}\text{div}(c_{i}\nabla\Phi)-u\nabla c_{i}&=D_{i}z_{i}c_{i}\Delta\Phi+D_{i}z_{i}\nabla c_{i}\nabla\Phi-u\nabla c_{i}\\&=-\frac{D_{i}}{\varepsilon}z_{i}c_{i}\rho+D_{i}z_{i}\nabla c_{i}\nabla\Phi-u\nabla c_{i}.\end{split}\]
By $L^p$ theory for \eqref{2.3}, we get
\begin{equation}\label{new-2.10}
\|\Phi\|_{W^{2,p}(\Omega)}\leq C\Big(\sum_{i=1}^{N}\|c_i\|_{L^{p}(\Omega)}+\|h\|_{W^{2,p}(\Omega)}\Big),\end{equation}
which along with the Sobolev embedding yields
\begin{equation}\label{2.11-2}
\|\nabla\Phi\|_{\C^\alpha(\bar{\Omega})}\leq C\Big(\sum_{i=1}^{N}\|c_i\|_{L^{p}(\Omega)}+\|h\|_{W^{2,p}(\Omega)}\Big).\end{equation}
Thus,
\begin{equation}\label{new-2.11}
\begin{split}\|\nabla c_i\nabla\Phi\|_{L^{\infty}(\Omega)}&\leq \|\nabla c_i\|_{L^{\infty}(\Omega)}\|\nabla\Phi\|_{L^{\infty}(\Omega)}\\&\leq Ct^{-\frac{1}{2}-\frac{1}{p}-\alpha}M\Big(M+\|h\|_{W^{2,p}(\Omega)}\Big).
\end{split}\end{equation}
Hence,
\[\begin{split}
\|g\|_{L^{\infty}(\Omega)}&\leq
C\left(\|c_i\rho\|_{L^{\infty}(\Omega))}
+\|\nabla c_i\nabla\Phi\|_{L^{\infty}(\Omega)}
+\|u\|_{L^{\infty}(\Omega)}\|\nabla c_i\|_{L^{\infty}(\Omega)}\right)\\&\leq C\left(t^{-\frac{2}{p}}M^{2}+t^{-\frac{1}{2}-\frac{1}{p}-\alpha}M^{2}+t^{-\frac{1}{2}
-\frac{1}{p}-\alpha}M\|h\|_{W^{2,p}(\Omega)}
+t^{-\frac{1}{2}-\frac{1}{p}-\alpha}M^{2}\right)\\&
\leq C\left(t^{-\frac{1}{2}-\frac{1}{p}-\alpha}M^2+t^{-\frac{1}{2}
-\frac{1}{p}-\alpha}M\|h\|_{W^{2,p}(\Omega)}\right).\end{split}\]
Denote
\begin{equation}
\label{del1}
\delta=\frac{\frac{1}{2}-\frac{1}{p}-\alpha}
{\frac{3}{2}+\frac{1}{p}+\alpha} >0.
\end{equation}
 Direct calculations imply
 \begin{equation}
 \|g\|_{L^q(Q_T)} \leq C T^{\delta}
 (M^2+M\|h\|_{W^{2,p}(\Omega)}).
 \label{Lq-esti-g}
 \end{equation}
Next we show $d_{i}\in X$ and get the estimates. Lemma \ref{lemma2.1} implies
\begin{equation*}
\begin{split}
\left|\int_{0}^{t}\int_{\Omega}H(x,t-s;y)\cdot g(y,s)dyds\right|
&\leq C\int_{0}^{t}\int_{\Omega}(t-s)^{-1}e^{-\frac{|x-y|^{2}}{16(t-s)}}dy\cdot\| g(\cdot,s)\|_{L^{\infty}}ds \\
&\leq C\int_{0}^{t}s^{-\frac{1}{2}-\frac{1}{p}-\alpha}ds(M^{2}+M\| h\|_{W^{2,p}(\Omega)}).
\end{split}
\end{equation*}
Thus
\begin{equation*}
\begin{split}
\sup_{t\in(0,T),x\in\Omega}t^{\frac{1}{2}}
\left|\int_{0}^{t}\int_{\Omega}H(x,t;y)g(y,s)dyds\right|\leq
CT^{1-\frac{1}{p}-\alpha}C\big(M^{2}+M\| h\|_{W^{2,p}(\Omega)}\big).
\end{split}
\end{equation*}
Similaily we can get the same estimates for $L^{p}$ and $\C^{1,\alpha}$ norm. Since $$d_{i}=\int_{\Omega}H(x,t;y)c_{i}(0)(y)dy+\int_{0}^{t}\int_{\Omega}H(x,t-s;y)(g(s,y)
-\Delta \gamma_i ) dyds,$$ by Lemma \ref{lemma2.2}, we get $d_{i}\in X$ and
\begin{equation}
\| d_{i}\|_{X}\leq CT^{\delta}(M^{2}+M\| h\|_{W^{2,p}(\Omega)})
+   C\sum_{i=1}^N( \|c_i(0) \|_{L^p(\Omega)}+
  \|\gamma_i\|_{\C^{2,\alpha}(\bar{\Omega})}).
\label{c_i_esti}
\end{equation}
To solve for $v$, we introduce $\xi_{s}^{t}(x)$ as the path lines starting from $x\in\Omega$, i.e.
\begin{equation}\label{2.10}
\left\{\begin{array}{ll}
\dfrac{d\xi_{s}^{t}(x)}{ds}=u(\xi_t(x),t),s\leq 0\\
\xi_0(x)=x.
\end{array}\right.
\end{equation}
It is easy to see that along the path lines $\omega$ satisfies
\begin{equation}\label{2.11}
\left\{\begin{array}{ll}
\dfrac{d\omega}{dt}=-K\nabla^\bot_\xi\rho\cdot\nabla_\xi\Phi,\\
\omega(0)=\omega_0(x).
\end{array}\right.
\end{equation}
Noting that $u\mid_{\partial\Omega}=0$ and $\nabla\cdot u=0$, we have
$\xi_{s}^{t}(x)\in \Omega.$ According to the standard ODE theory, \eqref{2.10}-\eqref{2.11} has a unique regular solution $(\xi,\omega)$. Denote by $\theta$ the stream function, i.e.
\begin{equation}\label{2.15}
\left\{\begin{array}{ll}
-\Delta\theta=\omega,& x\in \Omega,\\
\theta=0,&x\in\partial\Omega.
\end{array}\right.
\end{equation}
Since $\omega=\partial_{x_1}v_{2}-\partial_{x_{2}}v_{1}$, thanks to the Schauder theory for \eqref{2.15}, it is easy to see that $v$ can be solved by
\[v_{1}=\partial_{x_{2}}\theta,\ v_{2}=-\partial_{x_1}\theta, \]
and $v$ satisfies
\begin{equation}\label{2.16}
\begin{split}
\nabla\cdot v=0 \text{ and } \|v\|_{\C([0,T];\C^{1,\alpha}(\bar{\Omega}))}\leq \|\theta\|_{\C([0,T];\C^{2,\alpha}(\bar{\Omega}))}\leq C\|\omega\|_{\C([0,T];\C^{\alpha}(\bar{\Omega}))}.
\end{split}\end{equation}

We next establish the H\"{o}lder estimate for $\omega$. Integrating \eqref{2.11} gives
\begin{equation}\label{new-2.16}
\omega(x,t)=\omega_0(\xi_{s}^{t}(x))-K\int_0^t\nabla^\bot_\xi\rho(\xi,s)
\cdot\nabla_\xi\Phi(\xi,s)ds.\end{equation}
Thus,
\begin{equation}\label{new-2.17}
\begin{split}
\|\omega\|_{\C([0,T];\C(\bar{\Omega}))}&\leq\|\omega_0\|_{\C(\bar{\Omega})}
+C\int_0^T\|\nabla\rho(\cdot,t)\|_{\C(\bar{\Omega})}
\|\nabla\Phi(\cdot,t)\|_{\C(\bar{\Omega})}dt\\
&\leq\|\omega_0\|_{\C(\bar{\Omega})}
+C\int_0^Tt^{-\frac{1}{2}-\frac{1}{p}-\alpha}M(M+\|h\|_{W^{2,p}})dt\\
&\leq\|\omega_0\|_{\C(\bar{\Omega})}+CT^{\frac{1}{2}-\frac{1}{p}-\alpha}M(M+\|h\|_{W^{2,p}}).
\end{split}\end{equation}

On the other hand, it follows from \eqref{2.10} that
\[\begin{split}
|\xi_{-t}^{t}(x)-\xi_{-t}^{t}(y)|&\leq|x-y|+\int_0^t|u(\xi_s(x),s)-u(\xi_s(y),s)|ds\\
&\leq|x-y|+\|u\|_{\C([0,T];\C^1(\bar{\Omega}))}\int_0^t|\xi_s(x)-\xi_s(y)|ds.\end{split}\]
By Gronwall's inequality,
\begin{equation}\label{2.12}
|\xi_{-t}^{t}(x)-\xi_{-t}^{t}(y)|\leq|x-y|e^{\|u\|_{\C([0,T];\C^1(\bar{\Omega}))}\cdot t}.
\end{equation}
Hence,
\begin{equation}\label{2.13}
\begin{split}|\omega(x,t)-\omega(y,t)|\leq&
|\omega_0(\xi_{-t}^{t}(x))-\omega_0(\xi_{-t}^{t}(y))|\\&+K\int_0^t|\nabla^\bot\rho(x,s)
\cdot\nabla\Phi(x,s)-\nabla^\bot\rho(y,s)
\cdot\nabla\Phi(y,s)|ds.\end{split}
\end{equation}
By \eqref{2.12}, the first term on the RHS of \eqref{2.13} satisfies
\begin{equation}\label{2.20}
|\omega_0(\xi_{-t}^{t}(x))-\omega_0(\xi_{-t}^{t}(y))|\leq\|\omega_0\|_{\C^\alpha}
|\xi_{-t}(x)-\xi_{-t}(y)|^\alpha\leq \|\omega_0\|_{\C^\alpha}e^{\alpha Mt}|x-y|^\alpha.\end{equation}
The second term on the RHS of \eqref{2.13} satisfies
\begin{equation}\label{2.21}
\begin{split}
&K\int_0^t|\nabla^\bot\rho(x,s)
\cdot\nabla\Phi(x,s)-\nabla^\bot\rho(y,s)
\cdot\nabla\Phi(y,s)|ds\\&\leq K\int_0^t|\nabla\Phi(x,s)||\nabla^\bot\rho(x,s)-\nabla^\bot\rho(y,s)|ds\\&\quad
+K\int_0^t|\nabla^\bot\rho(y,s)||\nabla\Phi(x,s)-\nabla\Phi(y,s)|ds\\
&\triangleq I+II.\end{split}
\end{equation}
By \eqref{2.11-2},
\begin{equation*}\begin{split}
I&\leq K\|\nabla\Phi\|_{\C(\overline{Q_T})}
\int_0^t\|\nabla\rho(\cdot,s)\|_{\C^\alpha(\bar{\Omega})}ds|x-y|^\alpha\\&
\leq C\Big(M+\|h\|_{W^{2,p}(\Omega)}\Big)
\int_0^t s^{-\frac{1}{2}-\frac{1}{p}-\alpha}Mds|x-y|^\alpha\\
&\leq C(M^2+\|h\|^2_{W^{2,p}})t^{\frac{1}{2}-\frac{1}{p}-\alpha}|x-y|^\alpha,\end{split}\end{equation*}
and
\begin{equation*}
\begin{split}
II&\leq K\int_0^t\|\nabla\rho(\cdot,s)\|_{\C(\bar{\Omega})}
\|\nabla\Phi(\cdot,s)\|_{\C^\alpha(\bar{\Omega})}ds|x-y|^\alpha\\&
\leq C\int_0^t s^{-\frac{1}{2}-\frac{1}{p}-\alpha}M(M+\|h\|_{W^{2,p}})ds|x-y|^\alpha\\
&\leq C t^{\frac{1}{2}-\frac{1}{p}-\alpha}(M^{2}+\|h\|^2_{W^{2,p}}).\end{split}
\end{equation*}
It hence follows from \eqref{new-2.17}, \eqref{2.13}-\eqref{2.21} that
\begin{equation}\label{2.14}
\|\omega\|_{\C([0,T];\C^{\alpha}(\bar{\Omega}))}\leq CT^{\frac{1}{2}-\frac{1}{p}-\alpha}
(M^2+\|h\|^{2}_{W^{2,p}})+
\|\omega_0\|_{{\C^\alpha}(\bar{\Omega})}
(1+e^{\alpha MT}).
\end{equation}
Substituting this inequality into \eqref{2.16}, we obtain
\begin{equation}\label{2.23}
\begin{split}
\|v\|_{\C([0,T];\C^{1,\alpha}(\bar{\Omega}))}\leq CT^{\frac{1}{2}-\frac{1}{p}-\alpha}
(M^2+\|h\|^{2}_{W^{2,p}})+
C\|u(0)\|_{\C^{1,\alpha}(\bar{\Omega})}
(1+e^{\alpha MT}).
\end{split}\end{equation}
Define $x_{1}=\xi_{t_{2}-t_{1}}^{t_{1}}(x)$, then by definition $w(x,t_{1})=w(x_{1},t_{2})$
we also have
\begin{equation*}
\mid x_{1}-x\mid=\left|\int_{-t_{1}}^{-t_{2}}u(\xi_{s}^{t_{1}}(x),s+t_{1})ds\right|\leq \mid t_{1}-t_{2}\mid\cdot M
\end{equation*}
Note that
\begin{equation*}
w(x,t_{1})-w(x,t_{2})=w(x_{1},t_{2})-w(x,t_{2}).
\end{equation*}
 So we can get the   H\"{o}lder norms in time direction  using the estimates in spacial direction as above.
 More precisely,  we have.
\begin{equation}\label{2.23-2}
\begin{split}
\|w\|_{\C^{\alpha}(Q_T)}\leq CT^{\frac{1}{2}-\frac{1}{p}-\alpha}(M^2+\|h\|^{2}_{W^{2,p}})+C\|u(0)\|_{\C^{1,\alpha}(\bar{\Omega})}(1+e^{\alpha MT}).
\end{split}\end{equation}
 Using \eqref{2.16} we have
\begin{equation}\label{u_esti}
\begin{split}
\|v\|_{\C^{\alpha}(Q_T)}\leq CT^{\frac{1}{2}-\frac{1}{p}-\alpha}(M^2+\|h\|^{2}_{W^{2,p}})+C\|u(0)\|_{\C^{1,\alpha}(\bar{\Omega})}(1+e^{\alpha MT}).
\end{split}\end{equation}
In view of \eqref{c_i_esti}, \eqref{2.23} and \eqref{u_esti}, we now choose
\[M=4C(\|u(0)\|_{\C^{1,\alpha}(\bar{\Omega})})+
\sum_{i=1}^N (\|\gamma_i\|_{\C^{2,\alpha}(\bar{\Omega})}
+\|c_i(0)\|_{L^p(\Omega)}),\]
and $T\ll1$ such that
\[ \ CMT^{\delta}<\frac{1}{2}\text{ and } \alpha MT<\ln2,\]
 where $\delta$ is defined as in \eqref{del1}.
Then it is easy to see that $\Lambda$ is a mapping from $B_M$ to $B_M$.

\emph{Step 3. Existence.}
To get the existence, we use the Schauder fixed point theorem.  Pick a number $q_1$ such that $1<q_1<q$. Define  a space
$$Y=\C([0,T]; W^{1,q_1}(\Omega))\times
 \C([0,T]; \C^{1+\frac{\alpha}{2}}(\bar{\Omega})).$$
Then clearly $B_M$  is a compact convex set in $Y$.
It's also easy to check $\Lambda$ is continuous
as a mapping in $Y$.
 So we have a  continuous mapping in $Y$ which is
 invariant on a compact convex set.
 Now the existence follows from Schauder fixed point theorem.

\emph{Step 4. Uniqueness.} Let $(c_i^{(1)},u^{(1)})$ and $(c_i^{(2)},u^{(2)})$ be two solutions to system \eqref{1.1}-\eqref{1.7} satisfying \eqref{regularity}. We denote the differences by $(\bar{c}_i,\bar{u})=(c_i^{(1)}-c_i^{(2)},u^{(1)}-u^{(2)})$. Then $(\bar{c}_i,\bar{u})$ satisfies
\begin{equation*}
\partial_{t}\bar{c}_{i}+u^{(1)}\nabla \bar{c}_{i}+\bar{u}\nabla c_i^{(2)} =D_{i}\text{div}(\nabla \bar{c}_{i}+z_{i}c_{i}^{(1)}\nabla\bar{\Phi}+z_{i}\bar{c}_{i}\nabla\Phi^{(2)}),
\end{equation*}
\begin{equation*}
-\varepsilon\Delta\bar{\Phi}=\sum_{i=1}^{N} z_{i}\bar{c}_{i}=\bar{\rho},
\end{equation*}
\begin{equation*}
\partial_{t}\bar{u}+u^{(1)}\nabla \bar{u}+\bar{u}\cdot\nabla u^{(2)}+\nabla (p^{(1)}-p^{(2)})=-K(\rho^{(1)}\nabla\bar{\Phi}+\bar{\rho}\nabla\Phi^{(2)}),
\end{equation*}
\begin{equation*}
\nabla\cdot \bar{u}=0,
\end{equation*}
with $0$ initial-boundary conditions. Employing the $L^2$ theory for the Poisson equation, we have
\[\|\nabla\bar{\Phi}\|_{L^2(\Omega)}\leq C\sum_{i=1}^N\|\bar{c}_{i}\|_{L^2(\Omega)}.\]
Then the $L^2$ estimate for $(\bar{c}_i,\bar{u})$ gives
\[\begin{split}
&\frac{1}{2}\frac{d}{dt}\int_{\Omega}(|\bar{c}_{i}|^{2}+|\bar{u}|^{2})dx
+D_i\int_{\Omega}|\nabla\bar{c}_{i}|^{2}\\&=-\int_{\Omega}\bar{u}\nabla c_i^{(2)}\bar{c}_{i}-D_iz_i\int_{\Omega}(c_{i}^{(1)}\nabla\bar{\Phi}
+\bar{c}_{i}\nabla\Phi^{(2)})\nabla\bar{c}_{i}\\&\quad-\int_{\Omega}\bar{u}\cdot\nabla u^{(2)}\cdot\bar{u}-K\int_{\Omega}(\rho^{(1)}\nabla\bar{\Phi}
+\bar{\rho}\nabla\Phi^{(2)})\bar{u}\\
&\leq C(t^{-\frac{1}{2}-\frac{1}{p}-\alpha}M
+t^{-\frac{1}{p}}M+\|u^{(2)}\|_{\C([0,T];\C^{1})})
\int_{\Omega}(\sum_{i=1}^N|\bar{c}_{i}|^{2}+|\bar{u}|^{2})dx\\&\quad
+\frac{D_i}{4}\int_{\Omega}|\nabla\bar{c}_{i}|^{2}.
\end{split}\]
Thus, owing to Gronwall's inequality, we obtain $(\bar{c}_i,\bar{u})=(0,0)$, which completes the proof.
\end{proof}

\section{A priori estimates}\label{sec-3}

In this section we derive the a priori estimates that are necessary to prove the global well-posedness. To do this, first we need the maximum principle for strong solutions, whose proof is the same as that of Proposition 2 of \cite{CIL20-far}.
\begin{lemma}\label{maxmum}\cite{CIL20-far}
Let $(c_i,u)$ be a strong solution of system \eqref{1.1}-\eqref{1.7} on the time interval $[0,T]$. Assume that $c_i(0)\geq0$, $i=1,\cdots,N$. Then $c_i(x,t)\geq0$ for a.e. $(x,t)\in Q_T$.
\end{lemma}

Then we derive the energy estimate for system \eqref{1.1}-\eqref{1.7}.
\begin{lemma}[$L^2$ estimate]\label{lem-3.2}
Assume that $c_i(0)\in L^2(\Omega)$, $c_i(0)\geq0$, $i=1,\cdots,N$, and that $u(0)\in  L^2(\Omega)$. Suppose that one of the following two conditions holds:

i) $N=2$;

ii) $N\geq3$, $D_i\equiv D$ and $|z_i|\equiv z$, $i=1,\cdots,N$.

Then the solution $(c_{i},u)$ satisfies
\begin{equation}\label{3.0}
\max_{t\in[0,T]}\Big(\|u(\cdot,t)\|_{L^2}^2+\|\nabla\Phi(\cdot,t)\|_{L^2}^2
+\sum_{i=1}^N\|c_i(\cdot,t)\|_{L^2}^2\Big)+\int_0^T\sum_{i=1}^N\|\nabla c_i(\cdot,t)\|_{L^2}^2dt\leq C(T).
\end{equation}
\end{lemma}

\begin{proof}
Multiplying \eqref{1.3} by $u$ yields
\begin{equation}\label{3.1}
\frac{1}{2}\frac{d}{dt}\int_{\Omega}|u|^{2}dx+\int_{\Omega}(u\cdot\nabla u)\cdot udx+\int_{\Omega}\nabla p\cdot udx=-\int_{\Omega}K\rho u\cdot\nabla \Phi dx.
\end{equation}Performing an integration by parts, by \eqref{1.4} and \eqref{1.7}, we have
\begin{equation*}
\int_{\Omega}(u\cdot\nabla u)\cdot udx=\frac{1}{2}\int_{\Omega}u\cdot \nabla(|u|^{2})dx=-\frac{1}{2}\int_{\Omega}|u|^{2}\text{div} udx=0,\end{equation*}and
\begin{equation*}
\int_{\Omega}\nabla p\cdot udx=-\int_{\Omega}p ~\text{div} udx=0.
\end{equation*}
By the principle of superposition, we decompose $\Phi$  as $\Phi=\Phi_{0}+\Phi_h$, where
\begin{equation}\label{3.2}
\left\{\begin{array}{ll}
-\Delta\Phi_{0}=\displaystyle\frac{\rho}{\varepsilon}, &x\in\Omega,\\
\Phi_{0}=0,&x\in\partial \Omega,\\
\end{array}\right.
\end{equation}
and
\begin{equation}\label{3.3}
\left\{\begin{array}{ll}
-\Delta\Phi_h=0, &x\in\Omega,\\
\Phi_h =h(x),&x\in\partial \Omega.\\
\end{array}\right.
\end{equation}
It then follows from \eqref{3.1} that
\begin{equation}\label{3.4}
\frac{1}{2}\frac{d}{dt}\int_{\Omega}|u|^{2}dx=-\int_{\Omega}K\rho u\nabla\Phi_{0}dx-\int_{\Omega}K\rho u\nabla\Phi_h dx.
\end{equation}

Multiplying \eqref{1.1} by $z_{i}\Phi_{0}$ and summing in $i$, we get
\begin{equation}\label{3.5}
\begin{split}
&\sum_{i=1}^N\int_{\Omega}\partial_{t}c_{i} z_{i}\Phi_{0}dx+\sum_{i=1}^ND_{i}\int_{\Omega}z_{i}^{2}c_{i}|\nabla\Phi_{0}|^2dx\\
=&-\sum_{i=1}^N\int_{\Omega}u\cdot\nabla c_{i} z_{i}\Phi_{0}dx-\sum_{i=1}^ND_{i}\int_{\Omega}(z_{i}\nabla c_{i}\nabla \Phi_{0}+z_{i}^{2}c_{i}\nabla\Phi_h\nabla\Phi_{0})dx.
\end{split}
\end{equation}
By \eqref{1.2}, we have
\begin{equation}\label{3.9}
\sum_{i=1}^N\int_{\Omega}\partial_{t}c_{i}z_{i}\Phi_{0}dx
=-\varepsilon\int_\om\Delta\partial_t\Phi\Phi_{0}dx
=-\varepsilon\int_\om\Delta\partial_t\Phi_{0}\Phi_{0}dx
=\frac{\varepsilon}{2}\frac{d}{dt}\int_\om\mid\nabla\Phi_{0}\mid^{2}dx.\end{equation}
A simple calculation using $\text{div}u=0$ yields
\begin{equation}\label{new-3.10}
-\sum_{i=1}^N\int_\om u\cdot\nabla c_{i} z_{i}\Phi_{0}dx=-\int_\om u\nabla\rho\Phi_{0}dx=-\int_\om\nabla\cdot(\rho u)\Phi_{0}dx.
\end{equation}
By decomposing $c_i$ as $c_{i}=\tilde{c}_{i}+\gamma_{i}$, so that $\tilde{c}_{i}|_{\partial\Omega}=0$,
it follows from \eqref{3.5}-\eqref{new-3.10} that
\begin{equation}\label{3.8}
\begin{aligned}
&\frac{\varepsilon}{2}\frac{d}{dt}\int_\om\mid\nabla\Phi_0\mid^{2}dx
+\sum_{i=1}^ND_{i}\int_\om z_{i}^{2}c_{i}|\nabla\Phi_0|^2dx\\
=&-\int_\om\nabla\cdot(\rho u)\Phi_{0}dx-\sum_{i=1}^ND_{i}\int_\om (z_{i}^{2} c_{i}\nabla\Phi_h\nabla\Phi_0 +z_{i}\nabla\tilde{c}_{i}\nabla\Phi_{0}+ z_{i}\nabla\gamma_{i}\nabla\Phi_{0})dx.
\end{aligned}
\end{equation}
Combining $\eqref{3.8}\times K$ with \eqref{3.4}, noting
\[\int_\om(\nabla\cdot(\rho u)\Phi_{0}+\rho u\nabla\Phi_{0})dx=0,\]
we obtain
\begin{equation}\label{3.12}
\begin{split}\frac{1}{2}&\frac{d}{dt}\int_\om\Big(|u|^{2}
+K\varepsilon\mid\nabla\Phi_{0}\mid^{2}
\Big)dx+\sum_{i=1}^N KD_{i}z_{i}^{2}\int_\om c_{i}|\nabla\Phi_0|^2dx\\&=-\int_\om K\rho u\nabla\Phi_h dx-K\sum_{i=1}^ND_{i}\int_\om (z_{i}^{2}c_{i}\nabla\Phi_h\nabla\Phi_{0}+z_{i}\nabla \tilde{c}_{i}\nabla\Phi_{0}+z_{i}\nabla \gamma_{i}\nabla\Phi_{0})dx.
\end{split}\end{equation}
The RHS of \eqref{3.12} can be estimated as follows. By Schauder estimate for \eqref{3.3}, we have
\[\|\nabla\Phi_h\|_{C(\bar{\Omega})}\leq C\|h\|_{C^{2,\alpha}(\bar{\Omega})}\leq C. \]
It then follows that
\[\begin{split}
-\int_\om K\rho u\nabla\Phi_h dx&=-K\int_\om\sum_{i=1}^N z_ic_iu\nabla\Phi_hdx\\&
=-K\int_\om\sum_{i=1}^Nz_i\tilde{c}_{i}u\nabla\Phi_hdx
-K\int_\om\sum_{i=1}^Nz_i\gamma_{i}u\nabla\Phi_hdx\\
&\leq C\|\nabla\Phi_h\|_{L^\infty}
\int_\om(|u|^{2}+\tilde{c}_{i}^{2})dx+C\int_\om|u|^{2}dx+C\\
&\leq C\int_\om(|u|^{2}+\tilde{c}_{i}^{2})dx+C,\end{split}\]
and
\[\begin{split}
\Big|\int_\om z_{i}^{2}c_{i}\nabla\Phi_h\nabla\Phi_{0}dx\Big|&\leq\frac{1}{2}\int_\om z_{i}^{2}c_{i}|\nabla\Phi_{0}|^2dx+\frac{1}{2}\int_\om z_{i}^{2}c_{i}|\nabla\Phi_h|^2dx\\&\leq\frac{1}{2}\int_\om z_{i}^{2}c_{i}|\nabla\Phi_{0}|^2dx
+C\int_\om \tilde{c}_{i}^2dx+C.\end{split}\]
Similarly,
\[\begin{split}
-KD_i\int_\om (z_{i}\nabla\tilde{c}_{i}\nabla\Phi_{0}+ z_{i}\nabla\gamma_{i}\nabla\Phi_{0})dx
&=\frac{KD_i}{\varepsilon}\int_\om z_{i}\tilde{c}_{i}\rho dx- KD_i\int_\om z_{i}\nabla\gamma_{i}\nabla\Phi_{0}dx\\
&\leq C\sum_{i=1}^N\int_\om |\tilde{c}_{i}|^{2}dx+C\int_\om|\nabla\Phi_{0}|^2dx
+C.
\end{split}\]
Substituting these estimates into \eqref{3.12} leads to
\begin{equation}\label{new-3.13}
\begin{aligned}
&\frac{d}{dt}\int_\om\Big(|u|^{2}
+K\varepsilon\mid\nabla\Phi_{0}\mid^{2}
\Big)dx+\sum_{i=1}^N KD_{i}z_{i}^{2}\int_\om c_{i}|\nabla\Phi_0|^2dx\\
&\leq C\int_\om\left(|u|^{2}+\mid\nabla\Phi_{0}\mid^{2}+\tilde{c}_{i}^2
\right)dx+C.
\end{aligned}
\end{equation}

We next estimate $\tilde{c}_i$. Notice that $\tilde{c}_{i}$ satisfies
\begin{equation}\label{3.34}
\begin{aligned}
\partial_{t}\tilde{c}_{i}-D_{i}\Delta\tilde{c}_{i}
+u\nabla\tilde{c}_{i}-D_{i}\text{div}(z_{i}\tilde{c}_{i}\nabla\Phi)&
=D_{i}\Delta\gamma_{i}+
D_{i}z_{i}\nabla\gamma_{i}\nabla\Phi+D_{i}z_{i}\gamma_{i}\Delta\Phi-u\nabla\gamma_{i}\\
&\triangleq f_i.
\end{aligned}
\end{equation}
Multiplying \eqref{3.34} by $\frac{|z_i|\tilde{c}_{i}}{D_i}$,
and using $\text{div}u=0$, we get after summing in $i$
\begin{equation*}
\begin{aligned}
&\frac{1}{2}\frac{d}{dt}\int_\om\sum_{i=1}^N\frac{|z_i|}{D_i}\tilde{c}_{i}^2dx+\sum_{i=1}^N\int_\om |z_i| |\nabla\tilde{c}_{i}|^2dx\\&
=-\int_\om \sum_{i=1}^N|z_i|z_{i}\tilde{c}_{i}\nabla\Phi\nabla\tilde{c}_{i}dx+\int_\om \sum_{i=1}^N|z_i|\Delta\gamma_{i}\cdot\tilde{c}_{i}dx\\
&\quad-\int_\om \sum_{i=1}^N\Big(|z_i|z_{i}\gamma_{i}\nabla\Phi\nabla\tilde{c}_{i}+ \frac{|z_i|}{D_i} u\cdot\nabla\gamma_{i}\cdot\tilde{c}_{i}\Big)dx.
\end{aligned}
\end{equation*}
By Young's inequality, the last three terms on the RHS can be estimated as
\begin{equation*}
\begin{aligned}
&\Big|\int_\om |z_i|\Delta\gamma_{i}\cdot\tilde{c}_{i}dx\Big|+\Big|\int_\om |z_i|z_{i}\gamma_{i}\nabla\Phi\nabla\tilde{c}_{i}dx\Big|+ \Big|\int_\om\frac{|z_i|}{D_i} u\cdot\nabla\gamma_{i}\cdot\tilde{c}_{i}dx\Big|\\
&\leq\int_\om|z_i||\tilde{c}_{i}|^2dx+
\frac{1}{2}\int_\om|z_i||\nabla\tilde{c}_{i}|^2dx+C\int_\om(|\Delta\gamma_{i}|^2+\gamma_{i}^2|\nabla\Phi|^2+|\nabla\gamma_{i}|^2|u|^2)dx\\
&\leq\frac{1}{2}\int_\om|z_i||\nabla\tilde{c}_{i}|^2dx
+C\int_\om(|z_i||\tilde{c}_{i}|^2+|\nabla\Phi|^2+|u|^2)dx+C.
\end{aligned}
\end{equation*}
It then follows that
\begin{equation}\label{3.10}
\begin{aligned}
&\frac{d}{dt}\int_\om\sum_{i=1}^N\frac{|z_i|}{D_i}\tilde{c}_{i}^2dx+\sum_{i=1}^N\int_\om |z_i| |\nabla\tilde{c}_{i}|^2dx\\&
\leq-2\int_\om \sum_{i=1}^N|z_i|z_{i}\tilde{c}_{i}\nabla\Phi\nabla\tilde{c}_{i}dx
+C\int_\om(\sum_{i=1}^N|z_i||\tilde{c}_{i}|^2+|\nabla\Phi|^2+|u|^2)dx+C.
\end{aligned}
\end{equation}
The first term on the RHS of \eqref{3.10} satisfies
\begin{equation}\label{3.13}
\begin{aligned}
-2\int_\om \sum_{i=1}^N|z_i|z_{i}\tilde{c}_{i}\nabla\Phi\nabla\tilde{c}_{i}dx=\int_\om \sum_{i=1}^N |z_i|z_{i}\tilde{c}_{i}^2\Delta\Phi dx=-\frac{1}{\varepsilon}\int_\om \sum_{i=1}^N |z_i|z_{i}\tilde{c}_{i}^2\cdot\sum_{i=1}^Nz_ic_idx,
\end{aligned}
\end{equation}
where we have performed an integration by parts in the first equality, and used \eqref{1.2} in the second equality. In the case  $N=2$, we assume that $z_1>0$ and $z_2<0$ without loss of generality. Then it is easy to see that
\[\begin{split}
\sum_{i=1}^2|z_i|z_{i}\tilde{c}_{i}^2\cdot\sum_{i=1}^2z_ic_i
&=[z_1^2(c_1-\gamma_1)^2-z_2^2(c_2-\gamma_2)^2]\rho\\&
=[z_1(c_1-\gamma_1)+z_2(c_2-\gamma_2)][z_1(c_1-\gamma_1)-z_2(c_2-\gamma_2)]\rho\\&
=(\rho-z_1\gamma_1-z_2\gamma_2)(z_1c_1-z_2c_2-z_1\gamma_1+z_2\gamma_2)\rho\\&
=\rho^2(z_1c_1-z_2c_2)+\rho^2(-z_1\gamma_1+z_2\gamma_2)\\&\quad
-(z_1\gamma_1+z_2\gamma_2)(z_1c_1-z_2c_2)\rho+(z_1^2\gamma_1^2-z_2^2\gamma_2^2)\rho.
\end{split}\]
Noting $z_1c_1-z_2c_2\geq|\rho|$, it holds
\[\begin{split}
\sum_{i=1}^2|z_i|z_{i}\tilde{c}_{i}^2\cdot\sum_{i=1}^2z_ic_i
\geq|\rho|^3-C\rho^2-C|\rho|.\end{split}\]
Hence in the case  $N=2$, by \eqref{3.13} we have
\begin{equation*}
\begin{aligned}
-2\sum_{i=1}^2\int_\om |z_i|z_{i}\tilde{c}_{i}\nabla\Phi\nabla\tilde{c}_{i}dx
\leq -\frac{1}{\varepsilon}\int_\om |\rho|^3dx+C\sum_{i=1}^2\int_\om \tilde{c}_{i}^2dx+C.\end{aligned}
\end{equation*}
Substituting this inequality into \eqref{3.10} yields
\begin{equation}\label{3.11}
\begin{split}
&\frac{d}{dt}\int_\om\sum_{i=1}^2\frac{|z_i|}{D_i}\tilde{c}_{i}^2dx+\sum_{i=1}^2\int_\om |z_i| |\nabla\tilde{c}_{i}|^2dx+\frac{1}{\varepsilon}\int_\om |\rho|^3dx
\\&\leq C\int_\om(\sum_{i=1}^2|\tilde{c}_{i}|^2+|\nabla\Phi|^2+|u|^2)dx+C.
\end{split}
\end{equation}
One can thus get \eqref{3.0} in the case $N=2$  by combining \eqref{3.11} with \eqref{new-3.13} and using the Gronwall's inequality.

In the case $N\geq3$, as in \cite{L}, we set $E:=\sum_{i=1}^N\tilde{c}_{i}$ and $F:=\sum_{i=1}^Nz_i\tilde{c}_{i}$. In view of \eqref{3.34}, since $D_i\equiv D$ and $|z_i|\equiv z$ for $i=1,\cdots,N$, $E$ and $F$ satisfy
\begin{equation}\label{3.18}
\partial_tE-D\Delta E+u\nabla E-D\text{div}(F\nabla\Phi)=\sum_{i=1}^Nf_i,
\end{equation}
\begin{equation}\label{new-3.19}
\partial_tF-D\Delta F+u\nabla F-D\text{div}(z^2E\nabla\Phi)=\sum_{i=1}^Nzf_i.
\end{equation}
Combining $\eqref{3.18}\times z^2E$ with $\eqref{new-3.19}\times F$, and using $\text{div}u=0$, we have
\begin{equation*}\begin{split}
\frac{1}{2}\frac{d}{dt}\int_\Omega(z^2E^2+F^2)+D\int_\Omega(z^2|\nabla E|^2+|\nabla F|^2)=-\frac{D}{\varepsilon}\int_\om z^2EF\rho +\sum_{i=1}^N\int_\om f_i(z^2E+zF).
\end{split}\end{equation*}
Observing that
\[\begin{split}
z^2EF\rho&=\sum_{i=1}^N(z^2c_i-z^2\gamma_i)(\rho-\sum_{i=1}^Nz\gamma_i)\rho\\
&=\sum_{i=1}^Nz^2c_i\rho^2-\sum_{i=1}^Nz^2c_i\sum_{i=1}^Nz\gamma_i\rho
-\sum_{i=1}^Nz^2\gamma_i\rho+\sum_{i=1}^Nz^2\gamma_i\sum_{i=1}^Nz\gamma_i\rho\\
&\geq|z||\rho|^3-2\rho^2\sum_{i=1}^Nz^2|\gamma_i|
-|\rho|\sum_{i=1}^Nz^2|\gamma_i|\sum_{i=1}^Nz|\gamma_i|,
\end{split}\]
and
\begin{equation}\label{fi}
\int_\Omega f_i^2\leq C\int_\Omega\left(|u|^2+\rho^2+|\nabla\Phi_0|^2\right)+C,\end{equation}
we obtain
\begin{equation*}\begin{split}
&\frac{d}{dt}\int_\Omega(z^2E^2+F^2)+D\int_\Omega(z^2|\nabla E|^2+|\nabla F|^2)+\int_\Omega|\rho|^3\\&\leq C\int_\Omega(z^2E^2+F^2+|u|^{2}+\sum_{i=1}^N \tilde{c}_{i}^{2}+|\nabla\Phi_0|^2)+C.
\end{split}\end{equation*}
Combining this inequality with \eqref{new-3.13}, noting $$0\leq c_i\leq\sum_{i=1}^Nc_i=E+\sum_{i=1}^N\gamma_i,$$ by Gronwall's inequality, we get
\begin{equation}\label{new-3.20}
\max_{t\in[0,T]}\int_\om(|u|^{2}+\mid\nabla\Phi_{0}\mid^{2}
+\sum_{i=1}^N\tilde{c}_{i}^{2})dx+\int_0^T\int_\Omega|\rho|^3\leq C(T).\end{equation}
Owing to \eqref{3.13},
\begin{equation*}
\begin{aligned}
\left|\int_\om \sum_{i=1}^N|z_i|z_{i}\tilde{c}_{i}\nabla\Phi\nabla\tilde{c}_{i}dx\right|&\leq C\sum_{i=1}^N\int_\om \tilde{c}_{i}^2|\rho|\\
&\leq C\sum_{i=1}^N\|\tilde{c}_{i}\|_{L^2}\|\tilde{c}_{i}\|_{L^6}\|\rho\|_{L^3}\\
&\leq C\sum_{i=1}^N\|\nabla\tilde{c}_{i}\|_{L^2}\|\rho\|_{L^3}\\
&\leq \frac{1}{2}\sum_{i=1}^N\int_\om |z_i| |\nabla\tilde{c}_{i}|^2dx+\int_\Omega|\rho|^3+C.
\end{aligned}
\end{equation*}
Plugging this inequality into \eqref{3.10}, by  \eqref{new-3.20}, we have
\begin{equation*}
\sum_{i=1}^N\int_0^T\int_\om |\nabla\tilde{c}_{i}|^2\leq C(T).
\end{equation*}
The desired estimate \eqref{3.0} follows from this inequality and \eqref{new-3.20}.
\end{proof}

\begin{lemma}[$L^p$ estimate for $c_i$]\label{Lp}
Suppose that the same assumptions of Lemma \ref{lem-3.2} hold. Assume that $c_i(0)\in L^p(\Omega)$, $p>2$, then the solution satisfies
\begin{equation}\label{3.15}
\max_{t\in[0,T]}\|c_i(\cdot,t)\|_{L^p}\leq C(T),
\end{equation}
\begin{equation}\label{3.15-2}
\max_{t\in[0,T]}\|\nabla\Phi(\cdot,t)\|_{L^\infty}\leq C(T).
\end{equation}

\end{lemma}

\begin{proof}
Multiplying \eqref{3.34} by $\frac{z_i^2}{D_i}|\tilde{c}_{i}|^{p-2}\tilde{c}_{i}$, and using $\text{div}u=0$, we have after summing in $i$
\begin{equation}\label{3.16}
\begin{aligned}
&\frac{d}{dt}\sum_{i=1}^N\int_\om \frac{z_i^2}{3D_i}|\tilde{c}_{i}|^pdx+\sum_{i=1}^N\int_\om z_i^2 |\tilde{c}_{i}|^{p-2}|\nabla\tilde{c}_{i}|^2dx\\&
=-\int_\om \sum_{i=1}^Nz_{i}^3\tilde{c}_{i}\nabla\Phi\nabla(|\tilde{c}_{i}|^{p-2}\tilde{c}_{i})dx
+\int_\om \sum_{i=1}^Nz_i^2\Delta\gamma_{i}\cdot|\tilde{c}_{i}|^{p-2}\tilde{c}_{i}dx\\
&\quad-\int_\om\sum_{i=1}^N\Big( z_i^3 \gamma_{i}\nabla\Phi\nabla(|\tilde{c}_{i}|^{p-2}\tilde{c}_{i})+ \frac{z_i^2}{D_i} u\cdot\nabla\gamma_{i}\cdot|\tilde{c}_{i}|^{p-2}\tilde{c}_{i}\Big)dx.
\end{aligned}
\end{equation}
By $L^2$ estimate for \eqref{1.2}, Sobolev embedding theorem and Lemma \ref{lem-3.2}, we have
\[\|\nabla\Phi(\cdot,t)\|_{L^6}\leq C\|\Phi(\cdot,t)\|_{H^{2}}\leq C+C\|\rho(\cdot,t)\|_{L^2}\leq C(T).\]
Thus, the first term on the RHS of \eqref{3.16} satisfies
\begin{equation}\label{new-3.22}
\begin{aligned}
\left|\int_\om \sum_{i=1}^N z_{i}^3\tilde{c}_{i}\nabla\Phi\nabla(|\tilde{c}_{i}|^{p-2}\tilde{c}_{i})dx\right|
&\leq C\sum_{i=1}^N\int_\om z_i^2|\tilde{c}_{i}|^{\frac{p}{2}}||\nabla\tilde{c}_{i}|^{\frac{p}{2}}|\nabla\Phi|\\
&\leq C\sum_{i=1}^Nz_i^2\|\nabla\tilde{c}_i^{\frac{p}{2}}(\cdot,t)\|_{L^2}
\|\tilde{c}_i^{\frac{p}{2}}(\cdot,t)\|_{L^3}\|\nabla\Phi(\cdot,t)\|_{L^6}\\
&\leq C\sum_{i=1}^Nz_i^2\|\nabla\tilde{c}_i^{\frac{p}{2}}(\cdot,t)\|_{L^2}
\|\tilde{c}_i^{\frac{p}{2}}(\cdot,t)\|_{L^2}^{\frac{1}{2}}
\|\tilde{c}_i^{\frac{p}{2}}(\cdot,t)\|_{L^6}^{\frac{1}{2}}\\
&\leq C\sum_{i=1}^Nz_i^2\|\nabla\tilde{c}_i^{\frac{p}{2}}(\cdot,t)\|_{L^2}^{\frac{3}{2}}
\|\tilde{c}_i^{\frac{p}{2}}(\cdot,t)\|_{L^2}^{\frac{1}{2}}\\
&\leq \frac{1}{4}\sum_{i=1}^N\int_\om z_i^2 |\tilde{c}_{i}|^{p-2}|\nabla\tilde{c}_{i}|^2dx+C\sum_{i=1}^N\int_\om|\tilde{c}_{i}|^pdx,
\end{aligned}
\end{equation}
where we have used H\"{o}lder's inequality in the second inequality, interpolation in the third inequality and Sobolev embedding in the fourth inequality.

The third term on the RHS of \eqref{3.16} can be estimated as follows. By $L^p$ estimate for \eqref{1.2},
\[\|\Phi(\cdot,t)\|_{W^{2,p}}\leq C+C\|\rho(\cdot,t)\|_{L^p}\leq C+ C\sum_{i=1}^N\|c_i(\cdot,t)\|_{L^p}\leq C+ C\sum_{i=1}^N\|\tilde{c}_i(\cdot,t)\|_{L^p}.\]
Noting $p>2$, it then follows from the Sobolev embedding theorem that,
\begin{equation}\label{3.27}
\|\nabla\Phi(\cdot,t)\|_{L^\infty}\leq C+ C\sum_{i=1}^N\|\tilde{c}_i(\cdot,t)\|_{L^p},
\end{equation}
which gives
\begin{equation}\label{3.22}
\begin{split}
\Big|\int_\om z_i^3 \gamma_{i}\nabla\Phi\nabla(|\tilde{c}_{i}|^{p-2}\tilde{c}_{i})dx\Big|
&\leq Cz_i^2\int_\om  |\tilde{c}_{i}|^{p-2}|\nabla\Phi||\nabla\tilde{c}_{i}|dx\\
&\leq C z_i^2\|\tilde{c}_{i}^{\frac{p-2}{2}}\|_{L^2}\|\nabla\Phi\|_{L^\infty}
\|\tilde{c}_{i}^{\frac{p-2}{2}}\nabla\tilde{c}_{i}\|_{L^2}\\
&\leq Cz_i^2\|\tilde{c}_{i}\|_{L^p}^{\frac{p-2}{2}}\Big(1+\sum_{j=1}^N\|\tilde{c}_j\|_{L^p}\Big)\|\tilde{c}_{i}^{\frac{p-2}{2}}\nabla\tilde{c}_{i}\|_{L^2}\\
&\leq C\int_\om \sum_{i=1}^N|\tilde{c}_{i}|^pdx+\sum_{i=1}^N\frac{z_i^2}{4}\int_\om \tilde{c}_{i}^{p-2}|\nabla\tilde{c}_{i}|^2dx+C.
\end{split}\end{equation}
By Cauchy-Schwarz inequality, Lemma \ref{lem-3.2}, and Sobolev embedding theorem,
\begin{equation}\begin{split}\label{3.23}
\Big|\int_\om \frac{z_i^2}{D_i}u\cdot\nabla\gamma_{i}\cdot|\tilde{c}_{i}|^{p-2}\tilde{c}_{i}dx\Big|
&\leq C\left(\int_\om |u|^2\right)^{1/2}\left(\int_\om |\tilde{c}_{i}|^{2(p-1)}\right)^{1/2}\\&\leq
C\left(\int_\om |\tilde{c}_{i}|^{\frac{p}{2}\cdot\frac{4(p-1)}{p}}
\right)^{\frac{p}{4(p-1)}\cdot\frac{2(p-1)}{p}}\\
&\leq C\left(\int_\om |\nabla\tilde{c}_{i}^{\frac{p}{2}}|^2
\right)^{\frac{(p-1)}{p}}\\
&\leq\frac{z_i^2}{4}\int_\om |\tilde{c}_{i}|^{p-2}|\nabla\tilde{c}_{i}|^2+C.
\end{split}\end{equation}
Substituting \eqref{new-3.22}, \eqref{3.22} and \eqref{3.23} into \eqref{3.16}, by Lemma \ref{lem-3.2}, we obtain
\[\begin{split}\sum_{i=1}^N\int_\om \frac{z_i^2}{D_i}|\tilde{c}_{i}|^pdx+\sum_{i=1}^N\int_0^t\int_\om z_i^2 |\tilde{c}_{i}|^{p-2}|\nabla\tilde{c}_{i}|^2dx\leq C\int_0^t\int_\om \sum_{i=1}^N |\tilde{c}_{i}|^pdx+C.\end{split}\]
Thus, we get \eqref{3.15} from Gronwall's inequality. And \eqref{3.15-2} follows from \eqref{3.15} and \eqref{3.27}.
\end{proof}

\begin{lemma}[$L^2$ estimate for $\omega$]\label{lem-3.4} Suppose that the same assumptions of Lemma \ref{lem-3.2} hold.
Assume that $u(0)\in H^1(\Omega)$, then we have
\begin{equation}\label{3.19}
\max_{t\in[0,T]}\|\omega(\cdot,t)\|_{L^2}\leq C(T).
\end{equation}
\end{lemma}

\begin{proof}
Multiplying \eqref{omega} by $\omega$, using \eqref{1.4} and \eqref{1.7}, we have
\[\begin{split}
\frac{d}{dt}\int_\om\omega^2dx&=-K\sum_{i=1}^Nz_i\int_\om\nabla^\bot c_i \nabla\Phi \omega dx\\&\leq C\int_\om\omega^2dx+C\sum_{i=1}^N\int_\om|\nabla \Phi|^2|\nabla c_i|^2dx\\
&\leq C\int_\om\omega^2dx+C\sum_{i=1}^N\int_\om |\nabla c_i|^2dx,\end{split}\] where we have used Lemma \ref{Lp} in the last inequality.
The desired estimate \eqref{3.19}  then follows from Gronwall's inequality and  Lemma \ref{lem-3.2}.
\end{proof}

\begin{lemma}[$L^p$ estimate for $\omega$]\label{lem-3.6} Suppose that the same assumptions of Lemma \ref{lem-3.2} hold.
Assume that $u(0)\in W^{1,3}(\Omega)$ and that $c_i(0)-\gamma_i\in H_0^1(\Omega)$, then we have
\begin{equation}\label{3.31}
\max_{t\in[0,T]}\|\omega(\cdot,t)\|_{L^3}\leq C(T).
\end{equation}
\begin{equation}\label{3.29}
\begin{aligned}
\max_{t\in[0,T]}\int_\Omega|\nabla\tilde{c}_{i}|^2
+\int_0^T\int_\Omega|\Delta\tilde{c}_{i}|^2
\leq C(T).
\end{aligned}
\end{equation}

\end{lemma}

\begin{proof}
By Lemma \ref{lem-3.4} and the Sobolev embedding theorem,
\begin{equation}\label{3.33}
\|u\|_{C([0,T];L^r)}\leq C\|u\|_{C([0,T];H^1)}\leq C(T), \ \forall\ r\in(1,\infty).
\end{equation}
Multiplying \eqref{3.34} by $-\Delta\tilde{c}_{i}$, we have
\begin{equation}\label{3.35}
\begin{aligned}
\frac{1}{2}\frac{d}{dt}\int_\Omega|\nabla\tilde{c}_{i}|^2
+D_i\int_\Omega|\Delta\tilde{c}_{i}|^2
&=\int_\Omega(u\nabla\tilde{c}_{i}-D_{i}\text{div}(z_{i}\tilde{c}_{i}\nabla\Phi))
\Delta\tilde{c}_{i}-\int_\Omega f_i\Delta\tilde{c}_{i}\\&\leq\int_\Omega|u\nabla\tilde{c}_{i}\Delta\tilde{c}_{i}|
+C\int_\Omega|\nabla\tilde{c}_{i}|^2+C\int_\Omega f_i^2+\frac{D_i}{4}\int_\Omega|\Delta\tilde{c}_{i}|^2,
\end{aligned}
\end{equation}
where we have used Young's inequality and \eqref{3.15-2}. Recalling the inequality: $\forall\varepsilon>0$, $\exists C(\varepsilon)>0$ such that
\[\|\nabla g\|_{L^4}\leq \varepsilon\|\Delta g\|_{L^2}+C(\varepsilon)\|g\|_{L^2}, \forall g\in H^2(\Omega)\cap H_0^1(\Omega),\]
by H\"{o}lder's inequality and \eqref{3.33}, we get
\[\int_\Omega|u\nabla\tilde{c}_{i}\Delta\tilde{c}_{i}|\leq \|u\|_{L^4}\|\nabla\tilde{c}_{i}\|_{L^4}\|\Delta\tilde{c}_{i}\|_{L^2}\leq \varepsilon\|\Delta\tilde{c}_{i}\|_{L^2}^2+C(\varepsilon)\|\tilde{c}_{i}\|_{L^2}^2.\]
Substituting this inequality into \eqref{3.35} and choosing $\varepsilon\ll1$, we obtain
\begin{equation}\label{3.36}
\begin{aligned}
\int_\Omega|\nabla\tilde{c}_{i}|^2
+D_i\int_0^t\int_\Omega|\Delta\tilde{c}_{i}|^2
\leq \int_\Omega|\nabla\tilde{c}_{i}(0)|^2+C\int_0^t\int_\Omega f_i^2+C\int_0^t\int_\Omega |\tilde{c}_{i}|^2.
\end{aligned}
\end{equation} By \eqref{fi} and Lemma \ref{lem-3.2},
\[\begin{aligned}
\int_0^T\int_\Omega f_i^2\leq C\int_0^T\int_\Omega\Big(\sum_{i=1}^N|\tilde{c}_{i}|^2
+|u|^2+|\nabla\Phi|^2+C\Big)\leq C.\end{aligned}\]
It then follows from \eqref{3.36} that
\begin{equation}\label{3.37}
\begin{aligned}
\max_{t\in[0,T]}\int_\Omega|\nabla\tilde{c}_{i}|^2
+\int_0^T\int_\Omega|\Delta\tilde{c}_{i}|^2
\leq C(T).
\end{aligned}
\end{equation}

Multiplying \eqref{omega} by $|\omega|\omega$, by \eqref{3.15-2} and \eqref{3.19}, we have
\[\begin{split}
\frac{d}{3dt}\int_\om|\omega|^3dx&=-K\sum_{i=1}^Nz_i\int_\om\nabla^\bot c_i \nabla\Phi |\omega|\omega dx\\&\leq C\sum_{i=1}^N\int_\om|\omega|^2|\nabla \tilde{c}_{i}|dx +C\int_\om\omega^2dx\\
&\leq C\|\omega(\cdot,t)\|_{L^2}\|\omega(\cdot,t)\|_{L^3}\|\nabla\tilde{c}_{i}(\cdot,t)\|_{L^6}+C\\
&\leq C\int_\om|\omega|^3+\Big(\int_\om|\nabla \tilde{c}_{i}|^6\Big)^{\frac{1}{3}}+C\\
&\leq C\int_\om|\omega|^3+C\int_\om|\Delta \tilde{c}_{i}|^2+C\int_\om|\tilde{c}_{i}|^2+C.\end{split}\]
\eqref{3.31} then follows from Gronwall's inequality and \eqref{3.37}.
\end{proof}

\section{Global solution}\label{sec-4}
In this section, we show that the local solution obtained in section \ref{sec-2} can be extended to the global one by a contradiction argument.

\begin{proof}[Proof of Theorem \ref{thm-1}] Let $\T=\sup\{T\mid u\in \C([0,T];\C^{1,\alpha}(\bar{\Omega})),c_{i}\in W_q^{2,1}(Q_T)\}$. Assume that $\T<\infty$. By Lemma \ref{lem-3.6} and the Sobolev embedding theorem,
\begin{equation}\begin{split}
\|u\|_{\C(\overline{Q_{\T}})}&\leq C\|\omega\|_{\C([0,\T];L^3(\Omega))}\leq C,\\ \|\tilde{c}_{i}\|_{\C([0,\T];L^p(\Omega))}&\leq C\|\tilde{c}_{i}\|_{\C([0,\T];H^1(\Omega))}\leq C \ \text{ for any } 2<p<\infty.
\end{split}\end{equation}We rewrite \eqref{3.34} as
\begin{equation}\label{4.1}
\begin{aligned}
\partial_{t}\tilde{c}_{i}-D_{i}\Delta\tilde{c}_{i}
+(u-D_{i}z_{i}\nabla\Phi)\nabla\tilde{c}_{i}
=-\frac{D_{i}}{\varepsilon}z_{i}\tilde{c}_{i}\sum_{j=1}^Nz_jc_j+ f_i.
\end{aligned}
\end{equation}
Noting $f_i\in L^p(Q_{\T})$ for any $2<p<\infty$, applying the $L^p$ theory for \eqref{4.1}, we obtain $\tilde{c}_{i}\in W_p^{2,1}(Q_{\T})$, and
\[\|\tilde{c}_{i}\|_{\C([0,\T];W^{1,p}(\Omega))}\leq C.\]

To estimate $u$, as in the argument of local well-posedness, we note that along the path lines $\xi_t(x)$ defined by \eqref{2.10}, the vorticity $\omega$ satisfies \eqref{2.11}. Then as in \eqref{2.14}, we have
\[\|\omega\|_{\C([0,\T];\C^{\alpha}(\bar{\Omega}))}\leq C(\T).\]
Owing to the relation between $u$ and $\omega$ through the stream function $\theta$ defined by \eqref{2.15}, we have
\begin{equation*}
\| u\|_{\C([0,\T);\C^{1,\alpha}(\Omega))}\leq C(\T),
\end{equation*}
Notice for $\forall\ \varepsilon>0$
\begin{equation*}
\| u(\T-\varepsilon)\|_{\C^{1,\alpha}(\bar{\Omega})}+\| \tilde{c}_{i}(\T-\varepsilon)\|_{W^{1,p}(\Omega)}\leq C(\T),
\end{equation*}
where $C(\T)$ is independent of $\varepsilon$. Taking $(u(\T-\varepsilon),c_{i}(\T-\varepsilon))$ as the initial data, using the local well-posedness in Lemma \ref{local-wellposedness}, one can easily get a contradiction. Therefore, system \eqref{1.1}-\eqref{1.7} is globally well-posed.
\end{proof}

\section*{Acknowledgements}
Dapeng Du would like to thank Hongjie Dong  for wonderful discussions on the  fundamental solution of the heat equation. This work is supported by the Natural Science Foundation of Jilin Province (20210101144JC)
and National Natural Science Foundation of China under grant 11801067.

\end{document}